\documentclass[11pt]{amsart}

\usepackage{fullpage}
\usepackage{amscd}
\usepackage{amssymb, latexsym}
\usepackage[all,2cell,ps]{xy}
\usepackage{mathdots}

\theoremstyle{plain}
\newtheorem{thm}{Theorem}[section]
\newtheorem{theorem}[thm]{Theorem}

\newtheorem{proposition}[thm]{Proposition}

\newtheorem{claim}{Claim}

\theoremstyle{definition}
\newtheorem{de}[thm]{Definition}

\newtheorem{convention}[thm]{Convention}
\newtheorem{example}[thm]{Example}

\newcommand{\id}{\mathrm{id}}

\newcommand{\lmlt}[1]{\mathrm{LMlt}(#1)}

\newcommand{\ld}{\backslash}

\newcommand{\Ker}{\mathop{\mathrm{Ker}}}

\newcommand{\Mlt}{\mathop{\mathrm{Mlt}}}

\newcommand{\LMlt}{\mathop{\mathrm{LMlt}}}
\newcommand{\RMlt}{\mathop{\mathrm{RMlt}}}
\numberwithin{equation}{section}

\begin{document}

\title{The retraction relation for biracks}

\author{P\v remysl Jedli\v cka}
\author{Agata Pilitowska}
\author{Anna Zamojska-Dzienio}

\address{(P.J.) Department of Mathematics, Faculty of Engineering, Czech University of Life Sciences, Kam\'yck\'a 129, 16521 Praha 6, Czech Republic}
\address{(A.P., A.Z.) Faculty of Mathematics and Information Science, Warsaw University of Technology, Koszykowa 75, 00-662 Warsaw, Poland}

\email{(P.J.) jedlickap@tf.czu.cz}
\email{(A.P.) A.Pilitowska@mini.pw.edu.pl}
\email{(A.Z.) A.Zamojska-Dzienio@mini.pw.edu.pl}

\keywords{Yang-Baxter equation, set-theoretical solution, retraction of a solution, one-sided quasigroup, birack, congruence of an algebra.}
\subjclass[2010]{Primary: 16T25, 08A30. Secondary: 20N02, 08A62, 03C05.}

\date{\today}

\begin{abstract}
In \cite{ESS} Etingof, Schedler and Soloviev introduced, for each non-degenerate involutive set-theoretical solution $(X,\sigma,\tau)$ of the Yang-Baxter equation, the equivalence relation $\sim$ defined on the set $X$ and they considered a new non-degenerate involutive induced \emph{retraction} solution defined on the quotient set $X^{\sim}$.
It is well known that translating set-theoretical non-degenerate solutions of the Yang-Baxter equation into the universal algebra language we obtain an algebra called a \emph{birack}. In the paper we introduce the \emph{generalized retraction} relation $\approx$ on a birack, which is equal to $\sim$ in an involutive case. We present a complete algebraic proof that the relation $\approx$ is a congruence of the birack. Thus we show that the retraction of a set-theoretical non-degenerate solution is well defined not only in the involutive case but also in the case of all non-involutive solutions.
\end{abstract}

\maketitle

\section{Introduction}\label{sec1}
The Yang-Baxter equation is a fundamental equation occurring in integrable models in statistical
 mechanics and quantum field theory~\cite{Jimbo}.
Let $V$ be a vector space. A {\em solution of the Yang--Baxter equation}
is a linear mapping $r:V\otimes V\to V\otimes V$ such that
\begin{align*}
(id\otimes r) (r\otimes id) (id\otimes r)=(r\otimes id) (id\otimes r) (r\otimes id).
\end{align*}
Description of all possible solutions seems to be extremely difficult and therefore
there were some simplifications introduced (see e.g. \cite{Dr90}).

Let $X$ be a basis of the space $V$ and let $\sigma:X^2\to X$ and $\tau: X^2\to X$ be two mappings.
We say that $(X,\sigma,\tau)$ is a {\em set-theoretical solution of the Yang--Baxter equation} if
the mapping $x\otimes y \mapsto \sigma(x,y)\otimes \tau(x,y)$ extends to a solution of the Yang--Baxter
equation. It means that $r\colon X^2\to X^2$, where $r=(\sigma,\tau)$ satisfies the \emph{braid relation}:
\begin{equation}\label{eq:braid}
(id\times r)(r\times id)(id\times r)=(r\times id)(id\times r)(r\times id).
\end{equation}

A solution is called {\em non-degenerate} if the mappings $\sigma(x,\_)$ and $\tau(\_\,,y)$ are bijections,
for all $x,y\in X$.
A solution $(X,\sigma,\tau)$ is {\em involutive} if $r^2=\mathrm{id}_{X^2}$, and it is
\emph{square free} if $r(x,x)=(x,x)$, for every $x\in X$.

\begin{convention}
All solutions we study in this paper are set-theoretical and non-degenerate, so we will call them simply \emph{solutions}. The set $X$ can be of arbitrary cardinality. We investigate both involutive and non-involutive solutions.
\end{convention}

It is known (see e.g. \cite{S06, GIM08, D15}) that there is a one-to-one correspondence between (involutive) solutions of the Yang-Baxter equation and (involutive) \emph{biracks} $(X,\circ,\ld_{\circ},\bullet,/_{\bullet})$ -- algebras which have a structure of two one-sided quasigroups $(X,\circ,\ld_{\circ})$ and $(X,\bullet,/_{\bullet})$ and satisfy some additional identities \eqref{eq:b1}--\eqref{eq:b3}. This fact allows one to characterize solutions of the Yang-Baxter equation applying the universal algebra tools.

In \cite{ESS} Etingof, Schedler and Soloviev introduced, for each involutive solution $(X,\sigma,\tau)$, the equivalence relation $\sim$ on the set $X$: for each $x,y\in X$
\[
x\sim y\quad \Leftrightarrow\quad \tau(\_\,,x)=\tau(\_\,,y).
\]

Using properties of the \emph{structure group} $G(X,r):=\langle X\mid xy=\sigma(x,y)\tau(x,y)\; \forall x,y\in X\rangle$ of a solution $(X,\sigma,\tau)$ they argued that there was a natural induced involutive solution, called the \emph{retraction} of $(X,\sigma,\tau)$, defined on the quotient set $X^{\sim}$. Just recently, Lebed and Vendramin considered in \cite{LV} finite \emph{invertible} solutions which generalize the involutive ones by replacing the condition $r^2=\mathrm{id}_{X^2}$ by the assumption  that $r$ is a bijection.  In particular, in \cite[Lemma 7.4]{LV} they defined the relation $\sim$ on the set $X$: for each $x,y\in X$  
\begin{align}\label{rel:LV}
x\sim y\quad \Leftrightarrow\quad \sigma(x,\_\,)=\sigma(y,\_\,)\quad {\rm and}\quad \tau(\_\,,x)=\tau(\_\,,y),
\end{align}
and they showed that the mapping $r$ also induces a solution on the quotient set $X^{\sim}$.

In the language of biracks the fact that the induced solution is well defined on the quotient set simply means that the relation $\sim$ is a congruence of the corresponding birack. According to our knowledge, the direct proof of this fact was not presented anywhere.

In the paper we introduce the \emph{generalized retraction} relation $\approx$ on a birack $(X,\circ,\ld_{\circ},\bullet,/_{\bullet})$: for each $x,y\in X$
\[
x\mathrel{\approx} y \mathrel{\Leftrightarrow} \text{ for each } z\in X: \; x\circ z=y\circ z \text{ and } z\bullet x=z\bullet y,
\]
and we present a complete algebraic proof that the relation $\approx$ is a congruence of the birack (Theorem \ref{th:congr}). This fact generalizes the results of \cite{ESS} and \cite{LV}.

We also show in Subsection \ref{subsec:multi} that in involutive biracks the generalized retraction relation $\approx$ is equal to the relation $\sim$. Hereby our result confirms that the relation $\sim$ is a congruence of the corresponding involutive birack. What is interesting, our proofs are formulated in a pure universal algebra language and we do not need a structure group of the solution at all. In Section \ref{sec5} we give short and direct proofs in this manner of some results presented by Etingof et al. in \cite{ESS} or Rump in \cite{Rump}.

But we obtain even more. We show that the retraction of a solution is well defined also in the case of a non-involutive solution which means that the generalized retraction relation allows to study a new class of non-involutive solutions.

The paper is organized as follows. In Section \ref{sec2} we recall details of the connection between (involutive) biracks and (involutive) solutions of the Yang-Baxter equation to make the paper self-contained. In Section \ref{sec4a} we introduce the generalized retraction relation on a birack and prove that it is a congruence of it. Theorem \ref{th:congr} is the main result of the paper.
The proof of the theorem was found using the automated deduction software Prover9 \cite{Prover}.
Finally, in Section 4 we define the retraction of an arbitrary solution. We also show that the retraction relation plays an important role in different algebraic constructions, not only for biracks.

\section{Biracks and solutions of the Yang-Baxter equation}\label{sec2}
Let $X$ be a non-empty set and $\ast\colon X^2\to X$ be a binary operation. In universal algebra, a pair $(X,\ast)$ is called a
{\em groupoid} (or a {\em binar}, a {\em binary algebra}, a {\em magma}). In this paper we will also consider algebraic structures with a few
binary operations defined on a set $X$.
\vskip 3mm

Given a groupoid $(X,\ast)$ and an element $x\in X$, one can define two mappings: a {\em left translation} by $x$ and a {\em right translation} by $x$, respectively:
\begin{align*}
L_x\colon X\to X;\ a\mapsto x\ast a\quad\text{ and }\quad R_x\colon X\to X:\ a\mapsto a\ast x.
\end{align*}

\begin{de}
A groupoid $(X,\ast)$ is a {\em quasigroup} if $L_x$ and $R_x$ are bijections, for each $x\in X$.
\end{de}
In particular, this means that, for every $x,y\in X$, the equations $x\ast u=y$ and $u\ast x=y$ have unique solutions in $X$.
One may then define on $X$ two additional operations:
\begin{align*}
x\ld y:=L_x^{-1}(y), \quad {\rm and}\quad x/y:=R_y^{-1}(x)
\end{align*}
 of \emph{left division} and \emph{right division}, respectively and consider the quasigroup as an algebra
 $(X,\ast,\ld,/)$ with three binary operations satisfying for every $x,y\in X$ the following conditions:
\begin{equation}\label{eq:lq}
 x\ast(x\ld y)=y, \quad x\ld (x\ast y)=y,
 \end{equation}
 \begin{equation}\label{eq:rq}
  (y/ x)\ast x=y, \quad (y\ast x)/ x=y.
 \end{equation}

For example, every group $(G,\cdot)$ can be made into a quasigroup $(G,\ast,\ld,/)$ by taking $x\ast y=x\cdot y$, $x\ld y=x\cdot y^{-1}$ and $x/y=x^{-1}\cdot y$.

A quasigroup $(X,\ast)$ may be seen as a generalization of a group --- the operation~$\ast$ can be non-associative but its multiplication table, for a finite set $X$, is a {\em latin square}, that means in every row as well as in every column, each element appears exactly once. These algebras are widely studied and have many interesting applications, e.g. in cryptography, see \cite{Belousov, Pflug90, Shch17}.

\vskip 3mm

When studying quasigroups, one usually works with the permutation group
generated by $L_x$ and $R_x$. This group is called the {\em multiplication} group of~$X$
and denoted by~$\Mlt(X)$.

If all left translations in a groupoid $(X,\ast)$ are bijections
then the algebra is called a {\em left quasigroup}.
A {\em right quasigroup} is defined analogously. In one-sided quasigroups only one-sided multiplication groups are considered,
that means $\LMlt(X)=\langle L_x:\ x\in X\rangle$
for left quasigroups and $\RMlt(X)=\langle R_x:\ x\in X\rangle$ for right ones.

One can regard a left quasigroup $(X,\ast)$
  as an algebra $(X,\ast,\ld)$ with two binary operations  satisfying \eqref{eq:lq}
  and right quasigroup $(X,\ast)$ as an algebra $(X,\ast,/)$
  satisfying \eqref{eq:rq}. It is obvious that $(X,\ld,\ast)$
  is also a left quasigroup and $(X,/,\ast)$ is a right quasigroup.
However, even if we consider a one-sided quasigroup as an algebra with one basic binary operation,
  the second binary operation is implicitly present here anyway. There is an important reason for this --- it is well-known that the homomorphic image of a (one-sided) quasigroup defined with a single binary operation, need not be a (one-sided) quasigroup,
 see, e.g., \cite[Example~1]{Rump}. And our aim is to study a quotient of an algebraic structure with (one-sided) quasigroup operations. Roughly speaking, the class $\mathcal{A}$ of algebraic structures (of a given signature) is closed under the formation of homomorphic images if $\mathcal{A}$ is defined by a set of equations (famous Birkhoff's Theorem or {\bf HSP} Theorem).

A groupoid $(X,\ast)$ is \emph{idempotent}
if, for every $x\in X$,
\begin{align*}
L_x(x)=x,
\end{align*}
or equivalently, if for every $x\in X$,
\begin{align*}
x\ast x=x.
\end{align*}
In a left quasigroup $(X,\ast,\ld)$, the groupoid $(X,\ast)$
  is idempotent if and only if $(X,\ld)$ is idempotent. In this case we say that
 $(X,\ast,\ld)$ is idempotent.

\begin{example}\label{ex:lq}
Let $(A,\cdot)$ be a group, $c\in A$ be a fixed element and $f,g$ be two mutually inverse group automorphisms. Then algebraic structure $(A,\ast,\ld)$ with operations defined as follows
\begin{align*}
x\ast y&=x\cdot f(y\cdot x^{-1})\cdot c,\\
x\ld y&=g(y\cdot x^{-1})\cdot g(c^{-1})\cdot x,
\end{align*}
where $x^{-1}$ is an inverse element to $x$, is a left quasigroup. If $c$ is a neutral element in $A$, then $(A,\ast,\ld)$ is idempotent.
\end{example}
\begin{example}\label{ex:lq2}
Let $X=\{1,2,\ldots, n\}$ be a finite set and let $\pi_i$ for $i=1,2,\ldots n$ be its permutations (they needn't be different). Define on $X$ the operations
\begin{align*}
m\ast k&=\pi_m(k),\\
m\ld k&=\pi_m^{-1}(k).
\end{align*}
Then $(X,\ast,\ld)$ is a left quasigroup, and it is idempotent if and only if, for each $m\in X$, $\pi_m(m)=m$.
\end{example}

\vskip 3mm

Biracks are algebras that appear in low-dimensional topology. They are associated with a link diagram and they are
invariant (up to isomorphism) under the generalized Reidemeister moves for virtual knots
and links. On the other hand, they provide solutions to the Yang-Baxter equation \cite{FJSK}. The equational definition of a birack we use here was given first in \cite{S06}.
\begin{de}
An algebra $(X,\circ,\ld_{\circ},\bullet,/_{\bullet})$ with four binary operations is called a {\em birack}, if $(X,\circ,\ld_{\circ})$ is a left quasigroup,
$(X,\bullet,/_{\bullet})$ is a right quasigroup and the following holds for any $x,y,z\in X$:
\begin{align}
x\circ(y\circ z)=(x\circ y)\circ((x\bullet y)\circ z),\label{eq:b1}\\
(x\circ y)\bullet((x\bullet y)\circ z)=(x\bullet(y\circ z))\circ(y\bullet z), \label{eq:b2}\\
(x\bullet y)\bullet z= (x\bullet (y\circ z))\bullet (y \bullet z).\label{eq:b3}
  \end{align}
  \end{de}
A birack is {\em idempotent} if both one-sided quasigroups $(X,\circ,\ld_{\circ})$ and $(X,\bullet,/_{\bullet})$ are idempotent.
And it is {\em involutive} if it additionally satisfies for every $x,y\in X$:
 \begin{align}
 &(x\circ y)\circ(x\bullet y)=x\quad\Leftrightarrow\quad x\bullet y=(x\circ y)\ld_{\circ} x, \label{eq:linv}\\
&  (x\circ y)\bullet(x\bullet y)=y\quad \Leftrightarrow\quad x\circ y=y/_{\bullet}(x\bullet y).\label{eq:rinv}
 \end{align}
The identities \eqref{eq:linv} and \eqref{eq:rinv} are equivalent to another ones regarded by Rump \cite[Definition 1]{Rump} (see \eqref{eq:R16} and \eqref{eq:R17} below, and \cite[Definition 1.6.]{D15}), namely:
\begin{align}
&(x\circ y)\circ(x\bullet y)=x \quad \Leftrightarrow\quad x\bullet y=(x\circ y)\ld_{\circ} x \quad \Leftrightarrow\quad x=((x\circ y)\ld_{\circ} x)/_{\bullet} y
\quad \stackrel{\scriptsize(\ref{eq:lq})}\Leftrightarrow\quad\nonumber\\
&x=((x\circ y)\ld_{\circ} x)/_{\bullet} (x\ld_{\circ}(x\circ y)) \label{eq:R16a}.
\end{align}
Now, substituting $y$ by $x\ld_{\circ}y$ in \eqref{eq:R16a}, one obtains
\begin{equation}\label{eq:R16}
(y\ld_{\circ} x)/_{\bullet}(x\ld_{\circ} y)=x.
\end{equation}
On the other side, substitution of $y$ by $x\circ y$ in \eqref{eq:R16}, gives \eqref{eq:R16a}.
Similarly,
\begin{align}
&(x\circ y)\bullet(x\bullet y)=y \quad \Leftrightarrow\quad x\circ y=y/_{\bullet}(x\bullet y) \quad \Leftrightarrow\quad\nonumber\\
&x\ld_{\circ}(y/_{\bullet}(x\bullet y))=y\quad \Leftrightarrow\quad ((x\bullet y)/_{\bullet}y)\ld_{\circ}(y/_{\bullet}(x\bullet y))=y\quad \Leftrightarrow\quad \nonumber\\
&(x/_{\bullet} y)\ld_{\circ}(y/_{\bullet} x)=y.\label{eq:R17}
\end{align}
Taking $y=x$ in \eqref{eq:R16} and \eqref{eq:R17} we obtain that in an involutive birack for every $x\in X$ holds:
\begin{align}\label{eq:biq}
(x\ld_{\circ} x)/_{\bullet}(x\ld_{\circ} x)=x\quad{\rm and}\quad (x/_{\bullet} x)\ld_{\circ}(x/_{\bullet} x)=x,
\end{align}
which means that each involutive birack is a \emph{biquandle} (see \cite{FJSK}). Stanovsk\'y proved \cite[Lemma]{S06} that the identities \eqref{eq:biq} are equivalent in any birack.

In particular, the conditions \eqref{eq:biq} imply that in an involutive birack $(X,\circ,\ld_{\circ},\bullet,/_{\bullet})$ the mappings
\begin{align}\label{eq:T}
T\colon X\to X; \quad x\mapsto x\ld_{\circ} x,
\end{align}
and
\begin{align}
S\colon X\to X; \quad x\mapsto x/_{\bullet}x
\end{align}
are mutually inverse bijections. Recall that Etingof et al. in \cite[Proposition 1.4]{ESS} and Rump in \cite[Proposition 2]{Rump} obtained the same results but using the properties of the structure group for an involutive solution.
\vskip 3mm

Now we are ready to translate the constrain of a solution $(X,\sigma,\tau)$
of the Yang-Baxter equation into the language of the universal algebra.
For sake of simplicity let denote $\sigma(x,y)$ by  $x\circ y$ and $\tau(x,y)$ by $x \bullet y$. Then
\begin{align*}
r(x,y)=(\sigma(x,y),\tau(x,y))=(x\circ y,x\bullet y)=(L_x(y),{\textbf{\emph R}}_y(x)),
\end{align*}
where for each $s\in X$, $L_s\colon X\to X$ is the left translation with respect to the operation $\circ$,
and $\textbf{\emph R}_s\colon X\to X$ is the right translation with respect to the operation $\bullet$.

\vskip 2mm
This justifies the replacement of the notion $(X,\sigma,\tau)$, for a solution of the Yang-Baxter equation, by $(X,L,{\textbf{\emph R}})$, where $L\colon X\to X^X$; $x\mapsto L_x$ and ${\textbf{\emph R}}\colon X\to X^X$; $x\mapsto {\textbf{\emph R}}_x$. On the other hand, we can treat a solution $(X,L,{\textbf{\emph R}})$ as an algebra $(X,\circ,\ld_{\circ},\bullet,/_{\bullet})$
with four binary operations defined on the set $X$.

Using this notation, the braid relation \eqref{eq:braid} implies
that in the algebra $(X,\circ,\ld_{\circ},\bullet,/_{\bullet})$ the conditions \eqref{eq:b1}--\eqref{eq:b3}
hold (see \cite{D15, S06}).

Furthermore, the assumption that a solution $(X,L,{\textbf{\emph R}})$ is non-degenerate gives that, for every $s\in X$,
the mappings $L_s$ and $\textbf{\emph R}_s$ are invertible. Hence simply, $(X,\circ,\ld_{\circ})$ is a left quasigroup with
$x\ld_{\circ} y:=L_x^{-1}(y)$ and $(X,\bullet,/_{\bullet})$ is a right quasigroup with
$x/_{\bullet}y:=\textbf{\emph R}_y^{-1}(x)$.

Finally, a solution $(X,L,{\textbf{\emph R}})$ is involutive if $(X,\circ,\ld_{\circ},\bullet,/_{\bullet})$ satisfies \eqref{eq:linv} and \eqref{eq:rinv},
and it is square free if both one-sided quasigroups $(X,\circ,\ld_{\circ})$ and $(X,\bullet,/_{\bullet})$ are idempotent.

Hence, each (involutive) solution of the Yang-Baxter equation yields an (involutive) birack.

The converse is also true.
\begin{theorem}\label{th:ESS}\cite[Lemma 1.2]{D15}
If $(X,\circ,\ld_{\circ},\bullet,/_{\bullet})$ is an (involutive) birack, then defining $r(x,y)=(x\circ y, x\bullet y)$
we obtain an (involutive) solution of the Yang-Baxter equation.
\end{theorem}
By Theorem \ref{th:ESS} there is a one-to-one correspondence between
(involutive) solutions of the Yang-Baxter equation and (involutive) biracks.

\begin{example}[Lyubashenko, see \cite{Dr90}]
Let $X$ be a non-empty set and let $x\circ y=f(y)$, $x\bullet y=g(x)$, where $f,g\colon X\to X$. Then conditions \eqref{eq:b1}--\eqref{eq:b3} are satisfied if and only if $fg=gf$. If $f$ and $g$ are bijections, we can define two additional binary operations $x\ld_\circ y=f^{-1}(y)$, $x/_\bullet y=g^{-1}(x)$. The algebra $(X,\circ,\ld_{\circ},\bullet,/_{\bullet})$ is a birack which is involutive if and only if $g=f^{-1}$.

Such involutive birack corresponds to a solution which is called a \emph{permutation solution}. If $f=g=id$, the birack is called a \emph{projection} one and the corresponding solution is called \emph{trivial}.

\end{example}
\vskip 2mm
By Theorem \ref{th:ESS} we can see that not only involutive biracks are biquandles. By results of Smoktunowicz and Vendramin \cite[Corollary 3.3]{SV} biracks corresponding to solutions of the Yang-Baxter equation which originate from \emph{skew braces} are biquandles. Moreover, by observation of Lebed and Vendramin \cite[Lemma 1.4]{LV} any finite birack corresponding to an \emph{injective} solution (the canonical mapping $i\colon X\to G(X,r)$; $x\mapsto x$ is an injection)
is a biquandle, too.

Clearly, all involutive solutions are injective.

Many explicit examples of biracks can be found in \cite{BF,EN,Nel}. We use some of them, translating them into the notion used in this paper.
\begin{example}\cite[Example 10]{Nel}\label{ex:biracks1}
Consider a set $X=\{1,2,3,4\}$ with permutations $L_1=\textbf{\emph R}_2=(12)$, $L_2=\textbf{\emph R}_1=(12)(34)$, $L_3=L_4=\id$, $\textbf{\emph R}_3=\textbf{\emph R}_4=(34)$. Then, one can define operations $\circ$ and $\bullet$: $i\circ j:=L_i(j)$ and $i\bullet j:=\textbf{\emph R}_j(i)$, similarly as in Example \ref{ex:lq2}. In this way, one obtains a birack $(X,\circ,\ld_{\circ},\bullet,/_{\bullet})$ with the following multiplication tables for operations $\circ$ and $\bullet$ (note that here $\circ=\ld_{\circ}$ and $\bullet=/_{\bullet}$)
\[
  \begin{array}{c|cccc}
   \circ & 1 & 2 & 3 & 4\\
   \hline
   1 & 2 & 1 & 3 &4\\
   2 & 2 & 1 & 4 & 3 \\
   3 & 1 & 2 & 3 & 4\\
   4 & 1 & 2 & 3 &4
  \end{array}
  \qquad
	\begin{array}{c|cccc}
   \bullet & 1 & 2 & 3 & 4\\
   \hline
   1 & 2 & 2 & 1 &1\\
   2 & 1 & 1 & 2 & 2 \\
   3 & 4 & 3 & 4 & 4\\
   4 & 3 & 4 & 3 &3
  \end{array}.
 \]
This birack is non-idempotent, non-involutive, e.g. $(3\circ 1)\circ(3\bullet 1)=4$, and it is not a biquandle since $(3\ld _{\circ} 3)/_{\bullet}(3\ld _{\circ} 3)=(3\circ 3)\bullet (3\circ 3)=3\bullet 3=4$.
\end{example}

\section{Generalized retraction congruence}\label{sec4a}
One of the basic notions in general algebra is the one of a \emph{congruence} --- an equivalence relation on an algebraic structure compatible with this structure (the operations are well-defined on the equivalence classes). As the Fundamental Homomorphism Theorem (First Isomorphism Theorem) says, for a given algebraic structure its homomorphic images and quotients are exactly the same (up to isomorphism). For details, see any textbook on Universal Algebra e.g. \cite[Section 1.5]{Ber}. In case of biracks, necessary definitions look as follows.
\begin{de}
An equivalence relation $\theta$ on the set $X$ of elements of a birack $(X,\circ,\ld_{\circ},\bullet,/_{\bullet})$ is a \emph{congruence on} $(X,\circ,\ld_{\circ},\bullet,/_{\bullet})$ if it is compatible with all four operations of the birack $X$, i.e. if $x\mathrel{\theta} y$ and $z\mathrel{\theta} t$ then also
\begin{align*}
(x\circ z)\mathrel{\theta}(y\circ t)\\
(x\ld_{\circ} z)\mathrel{\theta}(y\ld_{\circ} t)\\
(x\bullet z)\mathrel{\theta}(y\bullet t)\\
(x/_{\bullet} z)\mathrel{\theta}(y/_{\bullet} t).
\end{align*}
\end{de}
If $\theta$ is a congruence on a birack $(X,\circ,\ld_{\circ},\bullet,/_{\bullet})$, then the quotient set $X^\theta=\{x^\theta: x\in X\}$ of the equivalence classes under $\theta$, is again a birack, called the \emph{quotient birack}, under operations defined by
\begin{align*}
x^\theta\circ z^\theta=(x\circ z)^\theta\\
x^\theta\ld_{\circ} z^\theta=(x\ld_{\circ} z)^\theta\\
x^\theta\bullet z^\theta=(x\bullet z)^\theta\\
x^\theta /_{\bullet} z^\theta=(x/_{\bullet} z)^\theta.
\end{align*}

\begin{de}
Let $(X,\circ,\ld_{\circ},\bullet,/_{\bullet})$ be a birack. The equivalence relation $\approx$ defined on $X$ in the following way
\begin{equation}\label{eq:ret}
x\mathrel{\approx} y \mathrel{\Leftrightarrow} L_x=L_y \text{ and } {\textbf{\emph R}}_x={\textbf{\emph R}}_y
\end{equation}
is called the \emph{generalized retraction}.
\end{de}

Obviously, the condition \eqref{eq:ret} can be formulated equivalently as
\begin{equation}\label{eq:reteq}
\forall\, z\in X \quad x\circ z=y\circ z \text{ and } z\bullet x=z\bullet y\quad
\end{equation}
or
\begin{equation}\label{eq:reteq2}
\forall\, z\in X \quad x\ld_\circ z=y\ld_\circ z\text{ and }z/_\bullet x=z/_\bullet y.
\end{equation}

\begin{theorem}\label{th:congr}
The generalized retraction is a congruence of a birack $(X,\circ,\ld_{\circ},\bullet,/_{\bullet})$.
\end{theorem}

\begin{proof}

Let $a,b,c,d$ be elements of a birack $(X,\circ,\ld_{\circ},\bullet,/_{\bullet})$ such that $a\mathrel{\approx} b$ and $c\mathrel{\approx} d$.
Now we show that $L_{a\circ c}=L_{b\circ d}$, ${\textbf{\emph R}}_{a\circ c}={\textbf{\emph R}}_{b\circ d}$, $L_{a\ld_\circ c}=L_{b\ld_\circ d}$  and ${\textbf{\emph R}}_{a\ld_\circ c}={\textbf{\emph R}}_{b\ld_\circ d}$. We start with a sequence of claims.
\vskip 2mm
For each $x,y,z\in X$ the following hold in a birack:

\begin{claim}
\begin{align}\label{claim1}
&y\bullet(x\ld_\circ z)=[(x/_\bullet y)\bullet(y\circ(x\ld_\circ z))]\ld_\circ[((x/_\bullet y)\circ y)\bullet z].
\end{align}
\end{claim}
\begin{proof}We have	
\begin{align*}
&[(x/_\bullet y)\bullet(y\circ(x\ld_\circ z))]\circ [y\bullet(x\ld_\circ z)]\stackrel{\scriptsize(\ref{eq:b2})}=
[(x/_\bullet y)\circ y]\bullet[((x/_\bullet y)\bullet y)\circ(x\ld_\circ z)]\stackrel{\scriptsize(\ref{eq:rq})}=\\
&=[(x/_\bullet y)\circ y]\bullet[x\circ(x\ld_\circ z)]\stackrel{\scriptsize(\ref{eq:lq})}=((x/_\bullet y)\circ y)\bullet z,
\end{align*}
and Claim~1 is obtained by dividing both sides with $[(x/_\bullet y)\bullet(y\circ(x\ld_\circ z))]$ from the left.
\end{proof}

\begin{claim}
\begin{align}\label{claim2}
&y\circ (x\ld_\circ z)=(x/_\bullet y)\ld_\circ[((x/_\bullet y)\circ y)\circ z].
\end{align}
\end{claim}
\begin{proof}
\begin{align*}
&(x/_\bullet y)\circ [y\circ (x\ld_\circ z)]\stackrel{\scriptsize(\ref{eq:b1})}=
[(x/_\bullet y)\circ y]\circ[((x/_\bullet y)\bullet y)\circ(x\ld_\circ z)]\stackrel{\scriptsize(\ref{eq:rq})}=\\
&=[(x/_\bullet y)\circ y]\circ[x\circ(x\ld_\circ z)]\stackrel{\scriptsize(\ref{eq:lq})}=((x/_\bullet y)\circ y)\circ z,
\end{align*}
and we divide both sides with $(x/_\bullet y)$.
\end{proof}

\begin{claim}
\begin{align}\label{claim3}
&(x\bullet y)\circ[(x\ld_\circ z)\bullet((x\ld_\circ z)\ld_\circ y)]=z\bullet[z\ld_\circ(x\circ y)].
\end{align}
\end{claim}
\begin{proof}
\begin{align*}
&(x\bullet y)\circ[(x\ld_\circ z)\bullet((x\ld_\circ z)\ld_\circ y)]\stackrel{\scriptsize(\ref{eq:lq})}=
[x\bullet ((x\ld_\circ z)\circ ((x\ld_\circ z)\ld_\circ y)]\circ[(x\ld_\circ z)\bullet((x\ld_\circ z)\ld_\circ y)]\stackrel{\scriptsize(\ref{eq:b2})}=\\
&=[x\circ(x\ld_\circ z)]\bullet[(x\bullet (x\ld_\circ z))\circ((x\ld_\circ z)\ld_\circ y)]\stackrel{\scriptsize(\ref{eq:lq})}=
z\bullet[(x\bullet (x\ld_\circ z))\circ((x\ld_\circ z)\ld_\circ y)]\stackrel{\scriptsize(\ref{eq:b1})}=\\
&=z\bullet[(x\circ(x\ld_\circ z))\ld_\circ(x\circ( (x\ld_\circ z)\circ((x\ld_\circ z)\ld_\circ y))]\stackrel{\scriptsize(\ref{eq:lq})}=
z\bullet[z\ld_\circ(x\circ y)],
\end{align*}
where (\ref{eq:b1}) is used in the form $(x\bullet y)\circ z=(x\circ y)\ld_\circ(x\circ(y\circ z))$.
\end{proof}

\begin{claim}
\begin{align}\label{claim4}
&x\bullet(x\ld_\circ c)=x\bullet(x\ld_\circ d).
\end{align}
\end{claim}
\begin{proof}
\begin{align*}
&x\bullet(x\ld_\circ c)\stackrel{\scriptsize(\ref{claim1})}=
[(x/_\bullet x)\bullet(x\circ(x\ld_\circ c))]\ld_\circ[((x/_\bullet x)\circ x)\bullet c]\stackrel{\scriptsize(\ref{eq:lq})}=\\
&=[(x/_\bullet x)\bullet c]\ld_\circ[((x/_\bullet x)\circ x)\bullet c]\stackrel{\scriptsize(\ref{eq:reteq})}=
[(x/_\bullet x)\bullet d]\ld_\circ[((x/_\bullet x)\circ x)\bullet d]\stackrel{\scriptsize(\ref{eq:lq}),	(\ref{claim1})}=x\bullet(x\ld_\circ d).\qedhere
\end{align*}
\end{proof}

\begin{claim}
\begin{align}\label{claim5}
&z\bullet(z\ld_\circ(x\circ c))=z\bullet(z\ld_\circ(x\circ d)).
\end{align}
\end{claim}
\begin{proof}
\begin{align*}
&z\bullet(z\ld_\circ(x\circ c))\stackrel{\scriptsize(\ref{claim3})}=(x\bullet c)\circ[(x\ld_\circ z)\bullet((x\ld_\circ z)\ld_\circ c)]
\stackrel{\scriptsize(\ref{eq:reteq})}=\\
&=(x\bullet d)\circ[(x\ld_\circ z)\bullet((x\ld_\circ z)\ld_\circ c)]\stackrel{\scriptsize(\ref{claim4})}=
(x\bullet d)\circ[(x\ld_\circ z)\bullet((x\ld_\circ z)\ld_\circ d)]\stackrel{\scriptsize(\ref{claim3})}=z\bullet(z\ld_\circ(x\circ d)).\qedhere
\end{align*}
\end{proof}

\begin{claim}
\begin{align}\label{claim6}
&(x/_\bullet y)\bullet(y\circ(x\ld_\circ c))=(x/_\bullet y)\bullet(y\circ(x\ld_\circ d)).
\end{align}
\end{claim}
\begin{proof}
\begin{align*}
&(x/_\bullet y)\bullet(y\circ(x\ld_\circ c))\stackrel{\scriptsize(\ref{claim2})}=
(x/_\bullet y)\bullet[(x/_\bullet y)\ld_\circ(((x/_\bullet y)\circ y)\circ c)]\stackrel{\scriptsize(\ref{claim5})}=\\
&=(x/_\bullet y)\bullet[(x/_\bullet y)\ld_\circ(((x/_\bullet y)\circ y)\circ d)]\stackrel{\scriptsize(\ref{claim2})}=
(x/_\bullet y)\bullet(y\circ(x\ld_\circ d)).\qedhere
\end{align*}
\end{proof}

\begin{claim}
\begin{align}\label{claim7}
b\circ ((a\ld_\circ z)\circ x)=z\circ((b\bullet (a\ld_\circ z))\circ x).
\end{align}
\end{claim}
\begin{proof}
\begin{align*}
b\circ((a\ld_\circ z)\circ x)&\stackrel{\scriptsize(\ref{eq:b1})}=(b\circ (a\ld_\circ z))\circ((b\bullet (a\ld_\circ z))\circ x)
\stackrel{\scriptsize(\ref{eq:reteq})}=\\
&=(a\circ (a\ld_\circ z))\circ((b\bullet (a\ld_\circ z))\circ x)\stackrel{\scriptsize(\ref{eq:lq})}
=z\circ((b\bullet (a\ld_\circ z))\circ x).\qedhere
\end{align*}
\end{proof}

\begin{claim}
\begin{align}\label{claim8}
b\bullet (a\ld_\circ c)=b\bullet (a\ld_\circ d).
\end{align}
\end{claim}
\begin{proof}
By \eqref{eq:b2}, the following condition holds
\begin{equation}\label{eq:c1}
y\bullet z=[x\bullet(y\circ z)]\ld_\circ [(x\circ y)\bullet ((x\bullet y)\circ z)].
\end{equation}
Hence,
\begin{align*}
b\bullet (a\ld_\circ x)&\stackrel{\scriptsize(\ref{eq:c1})}=[(b/_\bullet a)\bullet(b\circ (a\ld_\circ x))]\ld_\circ [((b/_\bullet a)\circ b)\bullet (((b/_\bullet a)\bullet b)\circ (a\ld_\circ x))]\stackrel{\scriptsize(\ref{eq:reteq})}=\\
&=[(b/_\bullet a)\bullet(a\circ (a\ld_\circ x))]\ld_\circ [((b/_\bullet a)\circ b)\bullet (((b/_\bullet a)\bullet a)\circ (a\ld_\circ x))]\stackrel{\scriptsize(\ref{eq:lq}),(\ref{eq:rq})}= \\
&=((b/_\bullet a)\bullet x)\ld_\circ [((b/_\bullet a)\circ b)\bullet ((b\circ (a\ld_\circ x))]\stackrel{\scriptsize(\ref{eq:reteq})}=\\
&=((b/_\bullet a)\bullet x)\ld_\circ [((b/_\bullet a)\circ b)\bullet ((a\circ (a\ld_\circ x))]\stackrel{\scriptsize(\ref{eq:lq})}=\\
&=((b/_\bullet a)\bullet x)\ld_\circ (((b/_\bullet a)\circ b)\bullet x).
\end{align*}
Hence,
\[
b\bullet (a\ld_\circ c)=((b/_\bullet a)\bullet c)\ld_\circ (((b/_\bullet a)\circ b)\bullet c)\stackrel{\scriptsize(\ref{eq:reteq})}=
((b/_\bullet a)\bullet d)\ld_\circ (((b/_\bullet a)\circ b)\bullet d)=b\bullet (a\ld_\circ d).
\qedhere \]
\end{proof}
\noindent
Summarizing, we obtain
\begin{align*}
&L_{a\circ c}(x)=(a\circ c)\circ x\stackrel{\scriptsize(\ref{eq:reteq})}=(b\circ c)\circ x\stackrel{\scriptsize(\ref{eq:lq})}=\\
&(b\circ c)\circ((b\bullet c)\circ ((b\bullet c)\ld_\circ x))\stackrel{\scriptsize(\ref{eq:b1})}=b\circ(c\circ((b\bullet c)\ld_\circ x))\stackrel{\scriptsize(\ref{eq:reteq})}=\\
&=b\circ (d\circ ((b\bullet d)\ld_\circ x))\stackrel{\scriptsize(\ref{eq:b1})}=(b\circ d)\circ((b\bullet d)\circ((b\bullet d)\ld_\circ x))\stackrel{\scriptsize(\ref{eq:lq})}=\\
&=(b\circ d)\circ x=L_{b\circ d}(x),
\end{align*}
and
\begin{align*}
&{\textbf{\emph R}}_{a\circ c}(x)=x\bullet(a\circ c)\stackrel{\scriptsize(\ref{eq:reteq})}=x\bullet(b\circ c)\stackrel{\scriptsize(\ref{eq:b3})}=\\
&=((x\bullet b)\bullet c)/_\bullet(b\bullet c)\stackrel{\scriptsize(\ref{eq:reteq})}=((x\bullet b)\bullet d)/_\bullet(b\bullet d)\stackrel{\scriptsize(\ref{eq:b3})}=\\
&=x\bullet (b\circ d)={\textbf{\emph R}}_{b\circ d}(x),
\end{align*}
which gives that $(a\circ c)\mathrel{\approx}(b\circ d)$.
\vskip 2mm
\noindent
Furthermore, one has
\begin{align*}
&L_b((a\ld_\circ d)\circ x)=b\circ((a\ld_\circ d)\circ x)\stackrel{\scriptsize(\ref{claim7})}=
d\circ((b\bullet (a\ld_\circ d))\circ x)\stackrel{\scriptsize(\ref{eq:reteq})}=\\
&=c\circ((b\bullet (a\ld_\circ d))\circ x)
\stackrel{\scriptsize(\ref{claim8})}=c\circ((b\bullet (a\ld_\circ c))\circ x)\stackrel{\scriptsize(\ref{claim7})}=\\
&=b\circ((a\ld_\circ c)\circ x)=L_b((a\ld_\circ c)\circ x).
\end{align*}
Since $L_b$ is a bijection, it follows that for each $x\in X$,
\begin{align}\label{eq:fl}
L_{L_a^{-1}(c)}(x)=L_{a\ld_\circ c}(x)=(a\ld_\circ c)\circ x=(a\ld_\circ d)\circ x=L_{a\ld_\circ d}(x)=L_{L_a^{-1}(d)}(x).
\end{align}
Moreover,
\begin{align*}
&y\bullet (x\ld_\circ c)\stackrel{\scriptsize(\ref{claim1})}=
[(x/_\bullet y)\bullet(y\circ(x\ld_\circ c))]\ld_\circ[((x/_\bullet y)\circ y)\bullet c]\stackrel{\scriptsize(\ref{eq:reteq})}=\\
&[(x/_\bullet y)\bullet(y\circ(x\ld_\circ c))]\ld_\circ[((x/_\bullet y)\circ y)\bullet d]\stackrel{\scriptsize(\ref{claim6})}=\\
&[(x/_\bullet y)\bullet(y\circ(x\ld_\circ d))]\ld_\circ[((x/_\bullet y)\circ y)\bullet d]\stackrel{\scriptsize(\ref{claim1})}=
y\bullet (x\ld_\circ d).
\end{align*}
Hence, for each $y\in X$
\begin{align}\label{eq:fr}
{\textbf{\emph R}}_{L_a^{-1}(c)}(y)={\textbf{\emph R}}_{a\ld_\circ c}(y)=y\bullet (a\ld_\circ c)=y\bullet (a\ld_\circ d)=
{\textbf{\emph R}}_{a\ld_\circ d}(y)={\textbf{\emph R}}_{L_a^{-1}(d)}(y).
\end{align}
By assumption, $L_a^{-1}=L_b^{-1}$. Then one gets
\[
L_{a\ld_\circ c}=L_{L_a^{-1}(c)}\stackrel{\scriptsize(\ref{eq:fl})}=L_{L_a^{-1}(d)}=L_{L_b^{-1}(d)}=L_{b\ld_\circ d},
\]
and
\[
{\textbf{\emph R}}_{a\ld_\circ c}={\textbf{\emph R}}_{L_a^{-1}(c)}\stackrel{\scriptsize(\ref{eq:fr})}={\textbf{\emph R}}_{L_a^{-1}(d)}=
{\textbf{\emph R}}_{L_b^{-1}(d)}={\textbf{\emph R}}_{b\ld_\circ d}.
\]
This finally implies $(a\ld_\circ c)\mathrel{\approx}(b\ld_\circ d)$.
\vskip 2mm
\noindent
The proof $(a\bullet c)\mathrel{\approx}(b\bullet d)$ and $(a/_\bullet c)\mathrel{\approx}(b/_\bullet d)$,
is similar due to the symmetry of the defining identities of a birack.

\end{proof}
Theorem \ref{th:congr} immediately forces that each quotient birack of $(X,\circ,\ld_{\circ},\bullet,/_{\bullet})$ satisfies all equational
properties of $(X,\circ,\ld_{\circ},\bullet,/_{\bullet})$. In particular, if $(X,\circ,\ld_{\circ},\bullet,/_{\bullet})$ is involutive (square-free, right cyclic (see \eqref{eq:RC}), etc.) then obviously $(X/_{\approx},\circ,\ld_{\circ},\bullet,/_{\bullet})$ is involutive (square-free, right cyclic, etc.).

\vskip 2mm

The next example shows that both defining conditions of the relation $\approx$ are essential.
\begin{example}\label{ex:zle}
Let $(X,\circ,\ld_{\circ},\bullet,/_{\bullet})$ be a birack on the set $X=\{0,1,2,3,4\}$ with the following multiplication tables for operations $\circ$ and $\bullet$:
\[
  \begin{array}{c|ccccc}
   \circ & 0 & 1 & 2 & 3 & 4\\
   \hline
   0 & 0 & 2 & 1 & 4 & 3 \\
   1 & 0 & 2 & 1 & 4 & 3\\
   2 & 4 & 2 & 1 & 3 & 0\\
   3 & 4 & 2 & 1 & 3 & 0\\
   4 & 3 & 2 & 1 & 0 & 4
  \end{array}
  \qquad
	\begin{array}{c|ccccc}
   \bullet & 0 & 1 & 2 & 3 & 4\\
   \hline
   0 & 0 & 3 & 3 & 0 & 0\\
   1 & 1 & 2 & 2 & 1 & 1\\
   2 & 2 & 1 & 1 & 2 & 2\\
   3 & 3 & 0 & 0 & 3 & 3\\
   4 & 4 & 4 & 4 & 4 & 4
  \end{array}.
 \]
Clearly, $L_0=L_1=(12)(34)$ and $L_2=L_3=(04)(12)$, but ${\textbf{\emph R}}_0\neq{\textbf{\emph R}}_1$. In consequence, $L_{0\circ 2}=L_1\neq L_4= L_{1\circ 3}$.
\end{example}
Note that the birack in Example \ref{ex:zle} is not involutive, since e.g. $(0\circ 1)\bullet (0\bullet 1)=2$.
In Section \ref{sec5} we will discuss behavior of the relation of generalized retraction in the case of an involutive birack.

\begin{example}\label{ex:SV}
In \cite[Example 3.9]{SV} the authors consider a non-involutive, not square-free birack on the set $X=\{1,2,\ldots,8\}$ obtained from a skew-brace. Namely, $L_1=L_6=\id$, $L_2=L_5=(25)(47)$, $L_3=L_8=(38)(47)$, $L_4=L_7=(25)(38)$ and $\textbf{\emph R}_1=\textbf{\emph R}_4=\textbf{\emph R}_6=\textbf{\emph R}_7=\id$, $\textbf{\emph R}_2=\textbf{\emph R}_3=\textbf{\emph R}_5=\textbf{\emph R}_8=(25)(38)$. In particular, $L_i=L_i^{-1}$ and $\textbf{\emph R}_i=\textbf{\emph R}_i^{-1}$ for each $i\in X$. Then, there are 4 equivalence classes of a congruence $\approx$: $\{1,6\},\; \{2,5\},\; \{3,8\},\; \{4,7\}$. The quotient birack $(X/_{\approx},\circ,\ld_{\circ},\bullet,/_{\bullet})$ is isomorphic to the birack on the set $Y=\{a,b,c,d\}$ with $L_i=\textbf{\emph R}_i=\id$ for each $i\in Y$. The latter birack is both involutive and square-free.
\end{example}

\section{Applications to other constructions}\label{sec5}

Relations closely related to the generalized retraction congruence studied in Section \ref{sec4a} are well established in the literature. They play an important role in different constructions applied to algebraic structures. The aim of this section is to recall some of these congruences, mainly in the context of the solutions of the Yang-Baxter equation. We want to present a unified approach to them in the universal algebra language.

\subsection{The retraction of solutions}\label{subsec:multi}
Let $(X,\sigma,\tau)$ be a solution of the Yang-Baxter equation. Etingof et al. in \cite[Section 3.2]{ESS} defined the following relation on the set $X$
\begin{align}\label{eq:ESS}
x\sim y\quad \Leftrightarrow\quad \tau(\_\,,x)=\tau(\_\,,y)
\end{align}
and showed (\cite[Proposition 2.2]{ESS}) that for an involutive solution
\[\tau(\_\,,x)=\tau(\_\,,y)\quad \Leftrightarrow\quad \sigma(x,\_)=\sigma(y,\_).
\]
In the language of biracks this means that for an involutive case relations $\approx$ and $\sim$ are equal.
Below we present the direct proof of the equivalence
\begin{equation}\label{equiv}
L_x=L_y\quad \Leftrightarrow\quad {\textbf{\emph R}}_x={\textbf{\emph R}}_y
\end{equation}
which holds in any involutive birack. The idea is similar to the proof from \cite{ESS} but we don't use the construction of the structure group which simplifies and shortens the proof.

First, note that if $(X,\circ,\ld_{\circ},\bullet,/_{\bullet})$ is a birack then directly by \eqref{eq:b1} we have that for every $x,y\in X$
\begin{equation}\label{eq:sol1}
L_{x\circ y}L_{x\bullet y}=L_xL_y,
\end{equation}
which is equivalent to
\begin{equation}\label{eq:sol2}
L_{x\bullet y}^{-1}L_{x\circ y}^{-1}=L_y^{-1}L_x^{-1}.
\end{equation}
Let $(X,\circ,\ld_{\circ},\bullet,/_{\bullet})$ be an involutive birack. We show that  $L_a^{-1}T=T\textbf{\emph R}_a$ for any $a\in X$, where $T\colon X\rightarrow X$ is an invertible mapping defined by \eqref{eq:T}. Indeed,
\begin{align*}
L_a^{-1}T(c)=L_a^{-1}L_c^{-1}(c)\stackrel{\scriptsize(\ref{eq:sol2})}=L_{c\bullet a}^{-1}L_{c\circ a}^{-1}(c)\stackrel{\scriptsize(\ref{eq:rinv})}=L_{c\bullet a}^{-1}\textbf{\emph R}_a(c)=L_{\textbf{\emph R}_a(c)}^{-1}\textbf{\emph R}_a(c)=T\textbf{\emph R}_a(c),
\end{align*}
for any $c\in X$.

Etingof et al., using bijective 1-cocycles defined on the structure group of an involutive solution $(X,\sigma,\tau)$,  reasoned that the quotient set $X^{\sim}$ has a structure of an involutive solution  $(X^{\sim},\overline{\sigma},\overline{\tau})$  with $\overline{\sigma}(x^{\sim},y^{\sim})=\sigma(x,y)^{\sim}$ and $\overline{\tau}(x^{\sim},y^{\sim})=\tau(x,y)^{\sim}$ for $x^{\sim},y^{\sim}\in X^{\sim}$  and $x\in x^{\sim}, y\in y^{\sim}$. They called such solution the \emph{retraction} of $(X,\sigma,\tau)$ and denoted it by ${\rm Ret}(X,\sigma,\tau)$.

Now, it is obvious, by Theorem \ref{th:congr} and \eqref{equiv}, that $\sim$ is just a congruence on an involutive birack. Clearly, the birack corresponding to the retraction solution ${\rm Ret}(X,\sigma,\tau)$ is the quotient birack $(X^{\sim},\circ,\ld_{\circ},\bullet,/_{\bullet})$.

A similar kind of retraction procedure but with use of the relation \eqref{rel:LV} appeared in \cite[Section 7]{LV} for finite, non-degenerate, invertible solutions.

Theorem \ref{th:congr} directly shows that the retraction of $(X,\sigma,\tau)$ is well defined not only for involutive or invertible solutions. Applying the notion from Section \ref{sec4a} we can naturally extend the definition.
\begin{de}\label{ret}
Let $(X,L,{\textbf{\emph R}})$ be a solution of the Yang-Baxter equation. The solution $(X^{\approx},L,{\textbf{\emph R}})$  with $L_{x^{\approx}}(y^{\approx})=L_x(y)^{\approx}$ and $\textbf{\emph R}_{y^{\approx}}(x^{\approx})=\textbf{\emph R}_y(x)^{\approx}$ for $x^{\approx},y^{\approx}\in X^{\approx}$  and $x\in x^{\approx},\; y\in y^{\approx}$ is the \emph{retraction} solution of $(X,L,{\textbf{\emph R}})$.
\end{de}

Obviously, the birack corresponding to the retraction solution of $(X,L,{\textbf{\emph R}})$ is equal to the quotient birack $(X^{\approx},\circ,\ld_{\circ},\bullet,/_{\bullet})$. 

Among involutive solutions, an important role is played by \emph{multipermutation solutions}, see e.g. \cite{CJO10,GIC,Ven}.

Let $(X,\sigma,\tau)$ be an involutive solution. One defines \emph{iterated retraction} in the following way: ${\rm Ret}^0(X,\sigma,\tau):=(X,\sigma,\tau)$ and
${\rm Ret}^k(X,\sigma,\tau):={\rm Ret}({\rm Ret}^{k-1}(X,\sigma,\tau))$, for any natural number $k>1$.

A solution $(X,\sigma,\tau)$ is called a \emph{multipermutation solution of level $m$} if $m$ is the least nonnegative integer such that
\[
|{\rm Ret}^{m}(X,\sigma,\tau)|=1.
\]

In the language of an involutive birack $(X,\circ,\ld_{\circ},\bullet,/_{\bullet})$ this means that applying $m$ times the congruence $\sim$ to the subsequent quotient biracks, one obtains the one-element birack. By Definition \ref{ret} one can generalize the notion for non-involutive solutions.
\begin{example}
The solution $(X,L,\textbf{\emph R})$ corresponding to a birack given in Example \ref{ex:SV} is a non-involutive multipermutation solution of level 2. In \cite[Example 4.18.]{SV} the skew-brace which produces this solution is said to have a \emph{finite multipermutation level}.
\end{example}

\subsection{Cycle sets}
In \cite{Rump} Rump showed that there is a correspondence between involutive solutions of the Yang-Baxter equation and \emph{non-degenerate cycle sets}.
\begin{de}\cite{Rump}
A groupoid $(X,\odot)$ is a \emph{cycle set} if all left multiplications are invertible and the equation
\begin{equation}\label{eq:RC}
(x\odot y)\odot(x\odot z)=(y\odot x)\odot(y\odot z)
\end{equation}
holds for all $x,y,z\in X$. A cycle set is \emph{non-degenerate} if
\begin{align}\label{map:T}
the \; mapping \quad T\colon X\to X; \quad x\mapsto x\odot x \quad  is\; a \; bijection.
\end{align}
\end{de}
It follows that each cycle set can be regarded as a left quasigroup $(X,\odot,\ld_{\odot})$ which satisfies \emph{right cyclic law} \eqref{eq:RC} with respect to the first operation --- this law can be also formulated in the form $L_{x\odot y}L_x=L_{y\odot x}L_y$. If $(X,\odot,\ld_{\odot})$ satisfies both \eqref{eq:RC} and \eqref{map:T} with respect to the first operation, we call it a \emph{non-degenerate right cyclic} left-quasigroup.

Let a birack $(X,\circ,\ld_{\circ},\bullet,/_{\bullet})$ be involutive. Then
by \eqref{eq:sol1}, \eqref{eq:linv} and \eqref{eq:lq}  we obtain for every $x,y\in X$
\begin{align*}
L_{x\circ y}L_{(x\circ y)\ld_{\circ} x}=L_xL_y\quad \Leftrightarrow\quad L_{x}L_{x\ld_{\circ} y}=L_yL_{y\ld_{\circ} x},
\end{align*}
substituting $y$ by $x\ld_{\circ} y$ for $\Rightarrow$, and $y$ by $x\circ y$ for $\Leftarrow$.
Since
\begin{align*}
L_{x}L_{x\ld_{\circ} y}=L_yL_{y\ld_{\circ} x} \quad \Leftrightarrow\quad L_{x\ld_{\circ} y}^{-1}L_{x}^{-1}=L_{y\ld_{\circ} x}^{-1}L_y^{-1},
\end{align*}
one obtains that for every $x,y,z\in X$
\begin{align*}
(x\ld_{\circ} y)\ld_{\circ} (x\ld_{\circ} z)=(y\ld_{\circ} x)\ld_{\circ} (y\ld_{\circ} z).
\end{align*}
Recall that the mapping $T$ defined by \eqref{eq:T} is a bijection in an involutive birack. Therefore, $(X,\ld_{\circ},\circ)$ is a non-degenerate right cyclic left quasigroup (compare to \cite[Proposition 1]{Rump}).
\vskip 2mm

It is also true that each non-degenerate right cyclic left quasigroup $(X,\ld,\ast)$ determines an involutive birack. (Note that here $(X,\ld)$ is a cycle set.)
Introducing the second operation $\ast$ into the type (which plays the role of left division of $\ld$, i.e. $\ast$ is also a left-quasigroup operation but not necessarily right cyclic), one can formulate the result in the language of right cyclic left quasigroups $(X,\ld,\ast)$ and involutive biracks.
\begin{theorem}\cite[Proposition 2]{Rump}, \cite[Proposition 1.5]{D15}\label{th:RD}
Let $(X,\ld,\ast)$ be a non-degenerate right cyclic left quasigroup.
Then defining $x\circ y=x\ast y$, $x\ld_{\circ}y=x\ld y$, $x\bullet y=(x\ast y)\ld x$, and $x/_{\bullet} y=z$, where $z$ is the unique one such that $z\ld z=y\ast(x\ld x)$,
the algebra $(X,\circ,\ld_{\circ},\bullet,/_{\bullet})$ is an involutive birack.
\end{theorem}
Hence, finding all involutive solutions of the Yang-Baxter equation is equivalent to constructing all non-degenerate right cyclic left quasigroups.

\begin{example}\label{Ex:2}
 Let $X=\{1,2,3,4\}$ and let $\ld$ be the following operation:
 \[
  \begin{array}{c|cccc}
  \ld & 1 & 2 & 3 & 4\\
  \hline
  1 & 3 & 4 & 2 & 1\\
  2 & 3 & 4 & 2 & 1\\
  3 & 4 & 3 & 1 & 2\\
  4 & 4 & 3 & 1 & 2
  \end{array}
 \]
The groupoid $(X,\ld)$ is a cycle set satisfying~(\ref{map:T}) and the corresponding involutive solution $(X,L,\textbf{\emph R})$
is $L_1=L_2=\textbf{\emph R}_1=\textbf{\emph R}_2=(1423)$ and $L_3=L_4=\textbf{\emph R}_3=\textbf{\emph R}_4=(1324)$.
\end{example}
\begin{example}\label{Ex:1}
Let $(A,+)$ be an abelian group and $f$ its automorphism. By $(\id-f)$ we denote the endomorphism of $(A,+)$ defined by $x\mapsto x-f(x)$. Assume that $(\id-f)$ is nilpotent of degree 2 and $c\in \Ker(\id-f)$. Then $(A,\ld,\ast)$ with
\begin{align*}
 x\ld y=(\id-f)(x)+f(y)+c \;\; {\rm and}\;\;  x\ast y=f^{-1}(y-(\id-f)(x)-c)
\end{align*}
is a non-degenerate right cyclic left quasigroup. Since in this case, $f-\id=\id-f^{-1}$, we obtain that $(A,\circ,\ld_{\circ},\bullet,/_{\bullet})$, with $x\circ y=(\id-f)(x)+f(y)+c$ and
$x\bullet y=f^{-1}(x)+(\id-f^{-1})(y)-f^{-1}(c)$, is an involutive birack.
\end{example}

Rump showed in \cite[Theorem 2]{Rump} that each finite cycle set is non-degenerate. He introduced the notion of a \emph{cycle group} and used its properties to obtain this result.
Actually, he proved a stronger result than this; if we are interested in finite
cycle sets only, we
 can present a short and direct proof based on the properties of finite one-sided quasigroups. It is evident (see e.g. \cite[Section 8.6.]{RS})
that each finite left quasigroup has a finite left multiplication group
and therefore, taking $k=|\lmlt{X}|$,
we have, for $x,y\in X$,
\begin{equation*}
L_x^{-k}(y)=\underbrace{x\ld (x\ld(x\ldots \ld(x}_{k-\text{times}}\strut\ld y)\ldots))=y,
\end{equation*}
or equivalently
\begin{equation*}
L_x^k(y)=\underbrace{x\ast (x\ast(x\ldots \ast(x}_{k-\text{times}}\strut\ast y)\ldots))=y.
\end{equation*}
Due to above properties one can consider finite one-sided quasigroups as algebras
with one basic binary operation e.g. of multiplication $\ast$, the
operation $\ld$ is defined then by means of multiplication: $x\ld y= L_x^{k-1}(y)$.

Moreover, by \eqref{eq:lq} and \eqref{eq:RC}, the following holds for $x,y$ in a right cyclic left quasigroup $(X,\ld,\ast)$:
\begin{align}\label{eq:kon1}
x\ld (y\ld y)=(y\ld (y\ast x))\ld (y\ld y)=((y\ast x)\ld y)\ld ((y\ast x)\ld y).
\end{align}
(Compare to \cite[Lemma 1]{Rump}). Equivalently, the condition \eqref{eq:kon1} can be stated  that in a right cyclic left quasigroup $(X,\ld,\ast)$, for each $x,y\in X$, there exists an element $a\in X$ such that $x\ld (y\ld y)=a\ld a$. In particular, the latter holds for $y=x$.
\begin{proposition}\cite[Theorem 2]{Rump}
Each finite right cyclic left quasigroup is non-degenerate.
\end{proposition}
\begin{proof}
Let $(X,\ld,\ast)$ be a finite right cyclic left quasigroup. We will show that the mapping $T\colon X\to X; x\mapsto x\ld x $ is a surjection. Let $z\in X$ and $k=|\lmlt{X}|>2$,
\begin{equation*}
z=L_z^{-k}(z)=L_z^{-k+2}(z\ld(z\ld z))=L_z^{-k+2}(u_1\ld u_1)=L_z^{-k+3}(z\ld (u_1\ld u_1))=\ldots =u_{k-2}\ld u_{k-2}=T(u_{k-2}),
\end{equation*}
where $u_i$, for $i=1,\ldots,k-2$, is an element of $X$ defined as in \eqref{eq:kon1}. For $k\leq 2$ the result is immediate.
\end{proof}

\vskip 2mm
Rump also proved (\cite[Lemma 2]{Rump}) that, for any cycle set $(X,\ld)$,
the relation $\alpha\subseteq X\times X$
\begin{align}\label{rel1}
 x\mathrel{\alpha} y\quad \Leftrightarrow \quad \forall z\in X \; x\ld z=y\ld z
\end{align}
is a congruence of $(X,\ld)$. In the proof, Rump used an assumption that for each $x\in X$, the mapping $a\mapsto x\ld a$ is a surjection. Again, introducing the operation $\ast$ into the type, one can present an alternative proof in the language of right cyclic left quasigroups $(X,\ld,\ast)$.

Let $x,x',y,y',z\in X$ and $x\mathrel{\alpha} x'$ and $y\mathrel{\alpha} y'$. This means that $x\ld z=x'\ld z$ and $y\ld z=y'\ld z$ for every $z\in X$. Then
by \eqref{eq:lq} and right-cyclic law we have
\[
(x\ld y)\ld z=(x\ld y)\ld(x\ld(x\ast z))=(y\ld x)\ld(y\ld(x\ast z))=(y'\ld x)\ld(y'\ld(x\ast z))=(x\ld y')\ld(x\ld(x\ast z))=
(x'\ld y')\ld z.
\]
Hence,
$x\ld y\mathrel{\alpha} x'\ld y'$.

But the quotient $(X^\alpha,\ld)$ is again a cycle set only if a cycle set $(X,\ld)$ was non-degenerate,  (see \cite[Example 1, Proposition 10]{Rump}). The reason for this is that in general, for an arbitrary right cyclic left quasigroup $(X,\ld,\ast)$, the relation $\alpha$ does not need to be a congruence of $(X,\ast)$.

But, by Theorem \ref{th:RD}, the additional condition \eqref{map:T} guarantees that the left quasigroup $(X,\ld,\ast)$ defines the right quasigroup on~$X$ as well.
We denote the latter by $(X,/,\cdot)$. Note that the operation $/$ is called \emph{dual} to $\ld$ in \cite[Definition 1]{Rump}, and the algebra $(X,\ld,/)$ is called \emph{RLC-system} in \cite[Definition 1.6]{D15}. Moreover, $(X,/)$ is a \emph{left-cyclic} right quasigroup and operations $\ld$ and $/$ are connected by \eqref{eq:R16} and \eqref{eq:R17} (substituting $\ld_{\circ}$ by $\ld$ and $/_{\bullet}$ by $/$). This is equivalent to the fact that the mapping $S\colon X\rightarrow X; x\mapsto x/x$ is the inverse of $T$ and \eqref{equiv} holds. It follows that the relation $\alpha$ is a congruence of the algebra $(X,/)$, too.

In the involutive birack $(X,\ast,\ld,\cdot,/)$, the conditions \eqref{eq:linv}, \eqref{eq:kon1} and \eqref{eq:biq} imply
\[
y\cdot x=(y\ast x)\ld y=[((y\ast x)\ld y)\ld((y\ast x)\ld y)]/[((y\ast x)\ld y)\ld((y\ast x)\ld y)]=
[x\ld(y\ld y)]/[x\ld(y\ld y)].
\]

Since by \eqref{eq:rinv}, $x\ast y=y/(x\cdot y)$, we obtain
that the operations of the right multiplication $\cdot$ and the left multiplication $\ast$ of an involutive birack are expressed by the operations
of the left $\ld$ and right $/$ divisions. Since $\alpha$ is a congruence with respect to divisions, the relation $\alpha$ is also a congruence due to the additional operation $\ast$. This exactly means that satisfying the condition \eqref{map:T} by a right cyclic left quasigroup $(X,\ld,\ast)$ is sufficient for the relation $\alpha$ to be a congruence of $(X,\ld,\ast)$.

\vskip 3mm

\subsection{Groupoid modes}\label{subsec:modes}
Romanowska and Smith investigated in \cite[Section 8.4]{RS} \emph{groupoid modes} -- idempotent and medial groupoids.
A groupoid $(X,\ast)$ is
\emph{medial} if, for every $x,y,z,t\in X$,
\begin{align*}\label{eq:med}
(x\ast y)\ast (z\ast t)=(x\ast z)\ast (y\ast t)  \quad \Leftrightarrow\quad
L_{x\ast y}L_z=L_{x\ast z}L_y.
\end{align*}

For each natural number $k$ they defined the following
equivalence relation on a mode $(X,\ast)$. For $a,b\in X$,
\[
a\varrho_k b \quad \Leftrightarrow\quad\forall x\in X \;
ax^k:=(((a\ast\underbrace{x)\ast x)\ast\dots )\ast x}_{k-\text{times}}=(((b\ast\underbrace{x)\ast x)\ast\dots )\ast x}_{k-\text{times}}=:bx^k.
\]
They proved that each of the relations $\varrho_k$ is a congruence of any mode $(A,\ast)$. Clearly, $\varrho_1=\;\sim$.
\vskip 3mm

They also showed that each $\varrho_k$-class $(a^{\varrho_k},\ast)$ is $k$-reductive. A groupoid $(X,\ast)$ is
\emph{$k$-reductive} if for every $x,y\in X$,
\begin{align}\label{eq:med}
xy^k=y.
\end{align}
$1$-reductive groupoid is usually called a \emph{right zero (right trivial) semigroup or a \emph{(right) projection groupoid}.}
Romanowska and Smith proved that if a mode $(X,\ast)$ is $n$-reductive then
the congruences $\varrho_k$ form an increasing chain of relations
\[
id_X=\varrho_0\leq \varrho_1\leq\ldots \leq \varrho_{n-1}\leq \varrho_n=X\times X.
\]

Additionally, if $k<n$ then a quotient $(A^{\varrho_k},\ast)$ is $(n-k)$-reductive. This means (\cite[Theorem 5.7.4]{RS}) that the variety of $n$-reductive binary modes coincides with the Mal'cev product of the varieties of $k$-reductive and $(n-k)$-reductive groupoid modes. In particular, this gives a recursive method of constructing $n$-reductive binary modes from right zero semigroups.

\subsection{The strong retraction}\label{subsec:strong}
Ced\'{o}, Jespers and Okni\'{n}ski introduced in \cite{CJO10} a stronger version of retractability of an involutive solution $(X,\sigma,\tau)$. They defined a relation $\rho$ on $X$ as follows:
\[
x\rho y\quad \Leftrightarrow\quad L_x=L_y\quad {\rm and}\quad x \;{\rm and}\; y\;{\rm are\; in \; the \; same\; orbit\; under\; the \; action \; of}\; \LMlt(X)\; {\rm on} \; X.
\]
The \emph{strong retraction} of $(X,\sigma,\tau)$ is taken then as the induced solution $(X^{\rho},\overline{\sigma},\overline{\tau})$.
Cedo et al. applied the relation $\rho$ to solve a problem of Gateva-Ivanova and showed that each involutive square free solution of the Yang-Baxter equation with abelian left multiplication group $\LMlt(X)$ is strongly retractable \cite[Corollary 2.9]{CJO10}.

In particular, they showed \cite[Theorem 2.5]{CJO10} that in a non trivial involutive square free solution $(X,\sigma,\tau)$ with abelian left multiplication group $\LMlt(X)$, there exists an orbit $X_k$, under the action of  $\LMlt(X)$ on $X$, with at least two different elements $x, y\in X_k$ such that $x\sim y$.

Independently, in  \cite[Definition 6.3]{JPZ} a quandle with such property of orbits was called \emph{quasi-reductive}.
Recall, an idempotent left quasigroup $(X,\ast,\ld_{\ast})$ is called a \emph{quandle} if it is \emph{left distributive}, i.e. for every $x,y,z\in X$
\begin{align*}
L_x(y\ast z)=L_x(y)\ast L_x(z)\quad \Leftrightarrow\quad
x\ast(y\ast z)=(x\ast y)\ast(x\ast z).
\end{align*}
Note that quandles are closely related to solutions of the Yang-Baxter equation --- any injective derived solution is a quandle \cite{ESG}. In knot theory biquandles were introduced as generalization of quandles.

Jedli\v cka, Pilitowska and Zamojska-Dzienio in \cite{JPZ} classified subdirectly irreducible medial quandles. A quandle is \emph{subdirectly irreducible} if and only if the intersection of all its non-trivial congruences, called the \emph{monolith congruence}, is non-trivial. In particular, they showed \cite[Main Theorem 6.4]{JPZ} that every subdirectly irreducible medial quandle with more than two elements is either a quasigroup or it is quasi-reductive. In finite case this characterization is even more readable since a subdirectly irreducible finite medial quandle is either quasigroup or it is $k$-reductive for some natural $k$. For reductive case, the congruence $\rho$ plays the role of the monolith. However, the group $\LMlt(X)$ for a quandle, needn't be abelian.

\end{document}